\documentclass[12pt,reqno]{amsart}
\usepackage{amssymb}
\usepackage[usenames,dvipsnames]{color}
\usepackage{euscript}
\usepackage{multicol}
\usepackage{amssymb}
\usepackage{amsmath}
\usepackage{times} 
\usepackage{colordvi}
\usepackage{graphicx}              
\usepackage{amsmath}               
\usepackage{amsfonts}              
\usepackage{amsthm}                
\usepackage{enumitem}
\usepackage{amssymb}
\usepackage{url}
\usepackage{color}
\usepackage{mathrsfs}
\usepackage{mathtools}
\usepackage{mathabx}
\usepackage{tikz}
\usepackage{bbm}
\usepackage[ansinew]{inputenc}

\newtheorem{thm}{Theorem}
\newtheorem{lem}[thm]{Lemma}
\newtheorem{prop}[thm]{Proposition}
\newtheorem{cor}[thm]{Corollary}

\newtheorem{rem}[thm]{Remark}
\theoremstyle{definition}
\theoremstyle{definition}

\newcommand{\norm}[1]{\left\lVert#1\right\rVert} 
\newcommand{\RR}{\mathbb{R}}            

\newcommand{\CC}{\mathbb{C}}

\newcommand{\Intr}{\displaystyle\int}
\newcommand{\Sumn}{\displaystyle\sum}
\newcommand*\diff{\mathop{}\!\mathrm{d}}
\DeclareMathOperator{\sgn}{sgn}

\begin{document}

\title[Spectral inequalities for a class of integral operators]
{Spectral inequalities for a class of integral operators}

\author{Ari Laptev}
\address{Ari Laptev: Imperial College London \\ 180 Queen's Gate \\ London SW7 2AZ \\ UK }
\email{a.laptev@imperial.ac.uk}

\author{Andrei Velicu}
\address{Andrei Velicu: Imperial College London \\ 180 Queen's Gate \\ London SW7 2AZ \\ UK }
\email{a.velicu15@imperial.ac.uk}



\dedicatory{To Nina Nikolaevna with respect and admiration}

\begin{abstract}
We obtain inequalities for the Riesz means for the discrete spectrum of a class of self-adjoint compact integral operators. Such bounds imply some inequalities for the counting function of the Dirichlet boundary problem for the Laplace operator. The paper is an extension of the results previously obtained in \cite{L1}.
\end{abstract}

\date{}

\maketitle

\noindent
Let $\Omega\subset \mathbb R^d$, $d\ge1$, be a domain of finite measure, $|\Omega|<\infty$, and let $K(x)$, $x\in \mathbb R^d$, be a homogeneous function  of order $\alpha-d$, 
such that $0<\alpha<d$,
$$
K(t x) = t^{\alpha-d} K(x), \qquad t>0.
$$
Assuming $K(x) = \overline{K(-x)}$, we consider the self-adjoint integral operator $\mathcal K$ defined in $L^2(\Omega)$ by
\begin{equation}\label{K}
\mathcal K u(x) = \int_\Omega K(x-y) u(y)\, dy. 
\end{equation}
Let us introduce $\widehat K$ the Fourier transform of $K$ in the sense of theory of distributions 
\begin{equation}\label{hatK}
\widehat K(\xi) = \int_{\mathbb R^d}  e^{-i x \xi} K(x)\, dx.
\end{equation}
The general theory of homogeneous distributions (see for example \cite{H, GSh}) says that if $u\in \mathcal{S}'(\RR^d)$ is homogeneous of degree $q$, then $\widehat{u}$ is a homogeneous distribution of degree $-q-d$. In addition, if $u\in C^\infty(\RR^d\setminus \left\{ 0 \right\})$, then also $\widehat{u}\in C^\infty(\RR^d\setminus \left\{ 0 \right\})$.
Therefore $\widehat K$ is a real-valued homogeneous of order $-\alpha$ function.

\medskip
\noindent
The operator $\mathcal K$ is a compact self-adjoint operator in $L^2(\Omega)$  that might have positive and negative eigenvalues $\{s_k^{\pm}\}_{k=1}^\infty$ accumulating at zero.  The Riesz means of the operator $\mathcal{K}$ is defined as
$$ \sum_k( |s_k^\pm| - s)_+.$$
The aim of this paper is to give both upper and lower bounds for the Riesz means of the operator $\mathcal{K}$. 

\medskip
\noindent
The structure of the paper is as follows. In Section 1 we fine and prove an upper bound for the Riesz means of arbitrary compact integral convolution type operator in $L^2(\Omega)$, where $\Omega\subset\mathbb{R}^d$ is a domain of finite measure. In Section 2 we obtain the lower bound which is more involved only for homogeneous kernels and $\Omega$ being strictly convex. Section 3 presents some special cases and applications.

\section{The Upper Bound}

\medskip
\noindent
Let $Q$ be a distribution from $\mathcal S'(\mathbb{R}^d)$ such that its Fourier transform $\widehat Q\in L^1_{\rm loc}(\mathbb{R}^d)$, satisfies
$$
\widehat Q(\xi) = \int_{\mathbb{R}^d} Q(x) e^{-ix\xi} dx \to 0 \quad {\rm as} \quad |\xi|\to\infty.
$$
and its convolution kernel generates a compact operator in $L^2(\Omega)$
$$
\mathcal Q u(x) = \int_\Omega Q(x-y) \,u(y) \, dy.
$$

%---------------------
\begin{thm}\label{Th1}
Let $\Omega\subset\mathbb R^d$, $d\ge1$,  be a domain of finite measure.  
Then the following inequality holds for the Riesz means of the eigenvalues $\{s_k^\pm\}$ of the operator  $\mathcal Q$ 
\begin{equation}\label{ReiszQ}
\sum_k( |\lambda_k^\pm| - \lambda)_+ \le (2\pi)^{-d}\, |\Omega| \,  \int_{\mathbb R^d} (|\widehat Q(\xi)|-\lambda)_+\, d\xi.
\end{equation}
\end{thm} 
%---------------------

\begin{proof}
Let $\{\psi_k^\pm\}$ be the orthonormal system of eigenfunctions of the operator $\mathcal Q$ corresponding to the eigenvalues $\lambda^\pm_k$. Then by definition we have
\begin{multline*}
\sum_k( |\lambda_k^\pm| - \lambda)_+ =  \sum_k \left( \left|(\mathcal Q \mathcal \psi_k^\pm, \psi_k^\pm)\right|  - \lambda  \|\psi_k^\pm\|^2\right)_+ \\
=
\sum_k \left( \left| \int_{\Omega}\int_{\Omega} Q(x-y)\psi_k^{\pm}(y)  \overline{\psi_k^{\pm}(x)} \, dydx \right| - \lambda \|\psi_k^\pm\|^2\right)_+  .
\end{multline*} 
Extending $\psi_k^\pm$ by zero outside $\Omega$ and using the Plancherel theorem we obtain
\begin{multline*} 
\sum_k( |\lambda_k^\pm| - \lambda)_+ \\ 
= 
 \sum_k  (2\pi)^{-d} \, \left( \left| \int_{\mathbb R^d} \widehat{Q}(\xi) \, |\widehat{\psi_k^\pm}(\xi)|^2 d\xi \right| - \lambda\, \int_{\mathbb R^d}|\widehat{\psi_k^\pm}(\xi)|^2 \, d\xi\right)_+ \\
 \le
  \sum_k  (2\pi)^{-d} \, \left(  \int_{\mathbb R^d} \left( |\widehat{Q}(\xi)| - \lambda\right) \, |\widehat{\psi_k^\pm}(\xi)|^2 d\xi  \right)_+\\
  \le 
  \sum_k  (2\pi)^{-d} \, \int_{\mathbb R^d} \left( |\widehat{Q}(\xi)| - \lambda\right)_+ \, |\widehat{\psi_k^\pm}(\xi)|^2 d\xi .
\end{multline*} 
Let $e_\xi(x)=e^{i\xi\cdot x}$. We now use that $\{\psi_k^\pm\}$ is the orthonormal system of functions in $L^2(\Omega)$ and derive using Parseval's identity
$$
  \sum_k |\widehat{\psi_k^\pm}(\xi)|^2 =  \sum_k \left| \int_\Omega e^{-ix\xi} \psi_k^\pm\, dx \right|^2 = \|e_\xi\|^2 = |\Omega|.
$$
This finally implies 
$$
\sum_k( |\lambda_k^\pm| - \lambda)_+ \le (2\pi)^{-d} \, |\Omega| \, \int_{\mathbb R^d} (|\widehat Q(\xi)|-\lambda)_+\, d\xi.
$$
The proof is complete.
\end{proof}

\noindent
Let now $\mathcal Q = \mathcal K$ defined in \eqref{K}.
The next statement follows immediately from Theorem \ref{Th1} by changing variables in the integral in \eqref{ReiszQ} by substituting the homogeneous function $\widehat K(\xi)$ given by  \eqref{hatK}.
%---------------------
\begin{cor}\label{Cor1}
Let $\Omega\subset\mathbb R^d$, $d\ge1$,  be a domain of finite measure and let $0<\alpha<d$. 
Then the following inequality holds for the Riesz means of the eigenvalues of the operator $\mathcal K$ 
$$
\sum_k( |\lambda_k^\pm| - \lambda)_+ \le (2\pi)^{-d}\, |\Omega| \, \lambda^{1- \frac{d}{\alpha}}\, \int_{\mathbb R^d} (|\widehat K(\xi)|-1)_+\, d\xi.
$$
\end{cor} 

%---------------------
\noindent
Similarly we obtain the following result related to the Helmholtz operator.
%---------------------
\begin{cor}\label{Cor2}
Let $\varkappa\ge 0$, $\Omega\subset\mathbb R^d$, $d\ge1$,  be a domain of finite measure and let 
$$
\widehat Q (\xi)= \frac{1}{|\xi|^2 + \varkappa^2}.
$$
Then the eigenvalues of the operator $\mathcal Q$ satisfy the inequality
$$
\sum_k( |\lambda_k^\pm| - \lambda)_+ \le (2\pi)^{-d}\, |\Omega| \,\int_{\mathbb R^d} \left(\frac{1}{|\xi|^2 + \varkappa^2}-\lambda\right)_+\, d\xi.
$$
\end{cor}

\medskip
\section{The Lower Bound} 

\medskip
\noindent
Let $\Omega\subset \mathbb R^d$, $d\ge2$, be a strictly convex domain of finite measure, $|\Omega|<\infty$,  and suppose that there exists $P\in C^\infty(\RR^d)$ such that the boundary $\partial\Omega$ is given by $P(x)=0$, and that $|\nabla P|=1$ on $\partial\Omega$.

\medskip
\noindent
The proof of the lower bound is more involved and requires some geometric considerations about the domain $\Omega$. Due to convergence issues, we need to make the assumption $0<\alpha<d-1$ throughout this section.

\medskip
\noindent
Let $\mathbbm{1}_\Omega$ be the characteristic function of $\Omega$, and introduce the function
$$ \eta(z)=\Intr \mathbbm{1}_\Omega(z+y)\mathbbm{1}_\Omega(y)\diff y.$$
Geometrically, $\eta(z)$ measures the volume of intersection of $\Omega$ with its translation by a vector $z$.
We write $z=r\tilde{z}$, where $r=|z|$ and $\tilde{z}\in \mathbb{S}^{d-1}$, so we can consider $\eta$ as a function defined on $[0,\infty)\times \mathbb{S}^{d-1}$. We want to compute the first terms in the Taylor expansion of $\eta$ (in the sense of distributions) around $(0,\tilde{z})$. We have 
$$ \eta(r,\tilde{z})
=\eta(0,\tilde{z})+r\eta'_r(0,\tilde{z})+r^2\Intr_0^1 \eta''_r((1-t)r,\tilde{z}) \diff t.$$
It is clear that $\eta(0,\tilde{z})=|\Omega|$, and in \cite{L1} the second term is computed using the formula $\nabla\mathbbm{1}_\Omega(z)=\delta(P)\nabla P(z)$ (see \cite{GSh}):
\begin{align*}
\eta'_r(r,\tilde{z})
= \Intr \delta(P(z+y))(\nabla P(z+y),\tilde{z}) \mathbbm{1}_\Omega(y) \diff y.
\end{align*}
In order to compute this, we use the following fact about the composition of the Dirac delta function with another function, which holds for $f,g:\RR^d\to \RR$
$$ \Intr \delta(f(x))g(x) \diff x = \Intr_{f^{-1}(0)} \frac{g(x)}{|\nabla f(x)|} \diff\sigma,$$
where $\sigma$ is the surface measure on $f^{-1}(0)$. We then have
$$ \eta'_r(r,\tilde{z})
=\Intr_{L_z} \frac{(\nabla P(z+u),\tilde{z})}{|\nabla P(z+u)|} \diff \sigma(u)
=\Intr_{L_z} (\nabla P(z+u),\tilde{z}) \diff \sigma(u),$$
where $L_z$ is the intersection of the domain $\Omega$ with the surface $P(z+y)=0$ ($L_{(0,\tilde{z})}$ will be understood as a limit, and it will depend on the direction $\tilde{z}$), and $\sigma$ is the surface measure on $L_z$. For the second equality, we used the fact that $|\nabla P|=1$ on $\partial\Omega$.

\medskip
\noindent
We need the following geometric fact. Let $R_\Omega=\min\text{dist}(u_1,u_2)$ where the minimum is taken over all points $u_1,u_2\in\partial\Omega$ such that $(\nabla P(u_1),\nabla P(u_2))=-1$; in other words, $R_\Omega$ is the diameter of the largest sphere entirely contained in $\Omega$, or the maximum number with the property that $\Omega \cap (z+\Omega) \neq \emptyset$ for all $|z|<R_\Omega$. Then there exists a family of diffeomorphisms $T_z:\mathbb{S}_+^{d-1}\to L_z$ from a fixed hemisphere $\mathbb{S}_+^{d-1}$ onto the surface $L_z$, for $|z|<R_\Omega$, which is infinitely differentiable in $z$. 

\medskip
\noindent
This fact allows us to change variables and to obtain
$$ \eta'_r(r,\tilde{z})
= \Intr_{\mathbb{S}_+^{d-1}} (\nabla P(z+T_z\theta),\tilde{z}) J(T_z) \diff \sigma(\theta),$$
where $J(T_z)$ is the Jacobian determinant of $T_z$. Since the integrand above is a smooth function of $z$, this shows that $\eta$ is smooth on $[0,R_\Omega)\times \mathbb{S}^{d-1}$.

\medskip
\noindent
We can then write down the Taylor expansion of $\eta$ around $r=0$ in the form
\begin{equation}\label{eta}
 \eta(r,\tilde{z}) 
= |\Omega| + r A_\Omega(\tilde{z})+r^2 B_\Omega(r,\tilde{z}),
\end{equation}
where $A_\Omega$ is a smooth function on $\mathbb{S}^{d-1}$ and $B_\Omega(r,\tilde{z})=\Intr_0^1 \eta''_r((1-t)r,\tilde{z})t\diff t$ is smooth on $[0,R_\Omega)\times \mathbb{S}^{d-1}$.

%--------------
\begin{rem}
%Take $\Omega$ to be a ball of radius $1$. The volume of intersection of two balls in $\RR^d$ can be computed in terms of the regularised incomplete beta function $I(z;a,b)$ (see \cite{L2}), namely we obtain
%%
%\begin{equation*}
%\eta(z)=\frac{\pi^{d/2}}{\Gamma\left( \frac{d}{2}+1\right)} I \left(1-\frac{|z|^2}{4}; \frac{d+1}{2},\frac{1}{2}\right).
%\end{equation*} 
%%
%In this case we can compute explicitly the derivative $\eta'_r(0,\tilde{z})$ and we obtain
%%
%$$ A_\Omega(\tilde{z}) \equiv -\frac{1}{2\pi} |\mathbb{S}^d|,$$
%%
%so $A_\Omega$ is a negative constant proportional to the area of the unit sphere in $\RR^{d+1}$. 

Take $\Omega\subset \RR^d$ to be a ball of radius one. In this case, we have $P(x)=\frac{1}{2}(1-|x|^2)$ and for any $z$, $L_{(0,\tilde{z})}$ is the hemisphere of $\Omega$ centred around the vector $-\tilde{z}$. Using our previous computations we then have
$$ \eta'_r(0,\tilde{z}) = - \int_{L_{(0,\tilde{z})}} u\cdot \tilde{z} \diff\sigma_{d-1}(u).$$
Here $\sigma_{d-1}$ is the surface measure of the sphere $\mathbb{S}^{d-1}$. Due to the symmetry of $\Omega$, $A_\Omega(\tilde{z})=\eta'_r(0,\tilde{z})$ does not depend on $\tilde{z}$, so it is a constant. By making a convenient choice, we can then compute
\begin{align*}
A_\Omega &= - \int_{\mathbb{S}^{d-2}}  \int_0^{\pi/2} \cos(\varphi)\sin^{d-2}(\varphi) \diff\varphi \diff\sigma_{d-2}(u) 
\\
&= -\frac{1}{d-1}|\mathbb{S}^{d-2}|.
\end{align*}

\end{rem}
%--------------

\medskip
\noindent
Let $F(z)=\frac{1}{|\Omega|}K(z)|z|A_\Omega(\tilde{z})$, so $F$ is a homogeneous function of degree $d-\alpha+1$. By the general theory, $\widehat{F}$ is a homogeneous function of degree $-\alpha-1$. Since $\widehat{K}$ is also homogeneous of degree $-\alpha$, then there exist continuous functions $f,g:\mathbb{S}^{d-1}\to \CC$ such that
$$ \widehat{K}(\xi) = \frac{f(\tilde{\xi})}{|\xi|^\alpha} 
\quad \text{ and } \quad 
\widehat{F}(\xi) = \frac{g(\tilde{\xi})}{|\xi|^{\alpha+1}}.$$
Let 
$$ \gamma = \Intr_{\mathbb{S}^{d-1}} \sgn(f(\theta)) |f(\theta)|^\frac{d-\alpha-1}{\alpha} g(\theta)\diff\sigma(\theta).$$

\medskip
\noindent
We are now ready to state the main result of this section.

\begin{thm} \label{lowerbound}
Let $\Omega\subset \RR^d$ be a convex domain of finite measure and suppose that $0<\alpha<d-1$. Then we have the following lower bound for the Riesz means of the operator $\mathcal{K}$
\begin{multline*}
\sum_k( |\lambda_k^\pm| - \lambda)_+ 
\geq \frac{|\Omega|}{(2\pi)^d} \lambda^{1-\frac{d}{\alpha}}\Intr_{\RR^d} \left( |\widehat{K}(\xi)| - 1\right)_+ \diff \xi
\\
+ \frac{|\Omega|}{(2\pi)^d}\frac{\gamma}{d-\alpha-1} \lambda^{1-\frac{d-1}{\alpha}}
+ o\,(\lambda^{1-\frac{d-1}{\alpha}}),
\end{multline*}
as $\lambda\to 0$.
\end{thm}

\medskip
\noindent
Before we prove this Theorem, we need some auxiliary results. 

%--------------
\begin{prop} \label{fourier1}
Let $h:\RR^d\to\RR$ be a smooth function with support contained in the ball of radius $R$ centred at $0$, for some $R>0$ and let $v$ be a homogeneous function of order $\kappa -d$, 
$\kappa>0$. Then the Fourier transform $\widehat{vh}$ of the product $vh$ satisfies 
$$ \widehat{vh}(\xi)= \widehat{v} h(0)  + O(|\xi|^{-\kappa-1}), \qquad {\rm as} \quad |\xi|\to\infty. $$
\end{prop}
%-------------

\begin{proof}
Since $h$ is smooth we can consider the its Taylor expansion around zero with a remainder term. Each term of the expansion is a homogeneous function that has a weaker singularity at zero than $\kappa$. The Fourier transform of the product of the remainder term and $v$  decays to zero as fast as we like depending on the number of term in the Taylor expansion.
\end{proof}

\noindent
We now to apply Proposition \ref{fourier1} in the context of the Taylor expansion of $\eta$, where the function $B$ is only smooth on $[0,R_\Omega)\times \mathbb{S}^{d-1}$. In order to avoid this problem, we introduce a smooth even function $\varkappa:\RR^d\to\RR$ such that $0\leq \varkappa \leq 1$, $\varkappa(x)=1$ for $|x|\leq R_\Omega/2$, and $\varkappa(x)=0$ for all $|x|\geq R_\Omega$. Now the function $h=\varkappa B$ satisfies the conditions of Proposition \ref{fourier1}. 

\medskip
\noindent
Let $K_0=\varkappa K$ and consider the operator 
$$\mathcal{K}_0 u(x) = \Intr_\Omega K_0(x-y)u(y)\diff y.$$
This is a compact self-adjoint operator on $L^2(\Omega)$ with positive and negative eigenvalues $\left\{ \mu_k^\pm \right\}_{k=1}^\infty$ accumulating at $0$. 
The kernel of the operator $\mathcal{K}_0 - \mathcal{K}$ is smooth. Therefore, using  
 \cite{W} (see also \cite{RSS}), we see that the eigenvalues $\nu_n^\pm$ of the operator $\mathcal{K}_0-\mathcal{K}$ satisfy $\nu^\pm_n = o(n^{-l})$, for all $l>0$. Therefore  it is sufficient to prove our result for $\mathcal{K}_0$.

\begin{proof} [Proof of Theorem \ref{lowerbound}.]

Let $\left\{\phi_k^\pm\right\}$ be an orthonormal set of eigenfunctions of $\mathcal{K}_0$ corresponding to the eigenvalues $\mu^\pm_k$. Fix $\lambda >0$ and let $\varphi(x):=(|x|-\lambda)_+$. Then we have
\begin{align*}
\Sumn_k (|\mu_k^\pm|-\lambda)_+ 
&= \Sumn_k \varphi(\mu_k^\pm) 
= \Sumn_k \varphi(\mu_k^\pm) \norm{\phi_k^\pm}_2
\\
&= \frac{1}{(2\pi)^d}\Sumn_k \varphi(\mu_k^\pm) \Intr_{\RR^d} |\widehat{\phi_k^\pm}(\xi)|^2 \diff\xi 
\\
&= \frac{1}{(2\pi)^d}\Sumn_k \varphi(\mu_k^\pm) \Intr_{\RR^d} \Intr_{\RR^d}\Intr_{\RR^d} \phi_k^\pm(x)\overline{\phi_k^\pm(y)} e^{-i(x-y)\cdot \xi}\diff x\diff y \diff\xi.
\end{align*}
Recall that $e_\xi(x)=e^{ix\cdot\xi}$. Using the spectral theorem for compact self-adjoint operators, we obtain
\begin{align*}
\Sumn_k (|\mu_k^\pm|-\lambda)_+ 
&=\frac{1}{(2\pi)^d}\Sumn_k \varphi(\mu_k^\pm) \Intr_{\RR^d} |(\phi_k^\pm,e_\xi)|^2\diff\xi 
\\
&= \frac{1}{(2\pi)^d}\Intr_{\RR^d} \Intr_{\RR}\varphi(\mu)\diff (E_\mu e_\xi,e_\xi)\diff\xi,
\end{align*}
where $E_\mu$ is the spectral measure of $\mathcal{K}_0$. 

\medskip
\noindent
Since $\Intr \diff (E_\mu e_\xi, e_\xi)=|\Omega|$ for all $\xi\in\RR^N$, then $\frac{1}{|\Omega|}\diff (E_\mu e_\xi, e_\xi)$ is a probability measure. Because also $\varphi$ is convex, then we can apply Jensen's inequality to obtain
$$ \varphi\left(\Intr \mu \frac{1}{|\Omega|}\diff (E_\mu e_\xi, e_\xi)\right) 
\leq \frac{1}{|\Omega|}\Intr \varphi(\mu) \diff (E_\mu e_\xi, e_\xi).$$
But, by the spectral theorem again,
\begin{align*}
\Intr \mu \diff (E_\mu e_\xi, e_\xi)
&=(\mathcal{K}_0e_\xi,e_\xi)
= \Intr_{\Omega}\Intr_{\Omega} K_0(x-y)e^{-i(x-y)\cdot \xi} \diff y\diff x
\\
&=\Intr_{\RR^d} K_0(z)\eta(z) e^{-iz\cdot \xi} \diff z.
\end{align*}

\medskip
\noindent
Using the expansion of $\eta$, we have
$$ \Intr_{\RR^d} K_0(z)\eta(z) e^{-iz\cdot \xi} \diff z
= |\Omega|\widehat{K_0}(\xi)+\widehat{G_1}(\xi)+\widehat{G_2}(\xi),$$
where 
$$G_1(z)=K_0(z)|z|A(\tilde{z}) \quad \text{ and } \quad G_2(z)=K_0(z)|z|^2B(|z|,\tilde{z}).$$
We have
$$ \widehat{K_0}(\xi)
= \widehat{K}(\xi)+\Intr_{\RR^d} K(z)(1-\varkappa(z))e^{-iz\cdot \xi}\diff z.$$
Since $\varkappa=1$ near $0$, then $K(z)(1-\varkappa(z))$ is smooth on $\RR^d$, so, by integration by parts, the integral in this relation is $O(|\xi|^{-k})$ as $|\xi|\to\infty$, for all $k>0$. Similarly,
$$ \widehat{G_1}(\xi)=|\Omega|\widehat{F}(\xi)+O(|\xi|^{-k})$$
as $|\xi|\to\infty$, for all $k>0$. Finally, by Proposition \ref{fourier1}, we have $\widehat{G_2}(\xi)=O(|\xi|^{-\alpha-2})$. Putting all these together, we obtained that
\begin{align*}
\frac{1}{|\Omega|}\Intr_{\RR^d} K_0(z)\eta(z) e^{-iz\cdot \xi} \diff z
= \widehat{K}(\xi) + \widehat{F}(\xi) + G(\xi),
\end{align*}
where $G(\xi)=O(|\xi|^{-\alpha-2})$ as $|\xi|\to\infty$. 

\medskip
\noindent
Going back to the computations above, we have
\begin{align} \label{integral}
\Sumn_k (|\mu_k^\pm|-\lambda)_+ 
\geq \frac{|\Omega|}{(2\pi)^d} \Intr_{\RR^d} \varphi(\widehat{K}(\xi) + \widehat{F}(\xi) + G(\xi))\diff\xi.
\end{align}
Thus, we need to estimate 
\begin{align*}
I:=\Intr_{\RR^d} &\left( \left|\widehat{K}(\xi) + \widehat{F}(\xi) + G(\xi)\right| -\lambda\right)_+ \diff\xi 
\\
&= \Intr_{\RR^d} \left( \widehat{K}(\xi) + \widehat{F}(\xi) + G(\xi) -\lambda\right)_+ \diff\xi
\\
& \quad + \Intr_{\RR^d} \left( -(\widehat{K}(\xi) + \widehat{F}(\xi) + G(\xi)) -\lambda\right)_+ \diff\xi.
\end{align*} 
Denote the two integrals on the right hand side by $I_1$ and $I_2$, respectively. 

\medskip
\noindent
Since $G(\xi)=O(|\xi|^{-\alpha-2})$ as $|\xi|\to\infty$, there exist constants $M,C>0$ such that 
$$ |G(\xi)|\leq M |\xi|^{-\alpha-2} \quad \text{ for all } \quad |\xi|\geq C.$$
Let $B_C$ be the ball of radius $C$ centered at the origin. We can estimate the integral $I_1$ by splitting it into an integral over $B_C$ and an integral over its complement, and treating each term separately. The integral over $B_C$ can be bounded easily using the inequality $(X+Y)_+\geq X_+ - |Y|$, and we obtain
\begin{align} \label{int1}
\Intr_{B_C} &\left( \widehat{K}(\xi) + \widehat{F}(\xi) + G(\xi)-\lambda \right)_+ \diff\xi \nonumber
\\
&\qquad \geq \Intr_{B_C} \left( \widehat{K}(\xi) + \widehat{F}(\xi) -\lambda \right)_+ \diff\xi 
- \Intr_{B_C} |G(\xi)| \diff\xi 
\end{align}

\medskip
\noindent
Let $m=\max\left\{1, \displaystyle\sup_{\mathbb{S}^{d-1}}|f(\theta)|,\displaystyle\sup_{\mathbb{S}^{d-1}}|g(\theta)| \right\}$. Then, it can be easily checked that for $\lambda<1$ we have
\begin{align*}
\widehat{K}(\xi) + \widehat{F}(\xi)
= \frac{f(\tilde{\xi})}{|\xi|^\alpha} + \frac{g(\tilde{\xi})}{|\xi|^{\alpha+1}} 
< \lambda \quad \text{ for all } |\xi|>\left(\frac{2m}{\lambda}\right)^{1/\alpha}.
\end{align*}
Using this, we can also estimate the second term 
\begin{align} \label{int2}
\Intr_{\RR^d\setminus B_C} &\left( \widehat{K}(\xi) + \widehat{F}(\xi) + G(\xi) -\lambda\right)_+ \diff\xi \nonumber
\\
&\geq \Intr_{\RR^d\setminus B_C} \left( \widehat{K}(\xi) + \widehat{F}(\xi) - \frac{M}{|\xi|^{\alpha +2}} -\lambda\right)_+ \diff\xi \nonumber
\\
&\geq \Intr_{\RR^d\setminus B_C} \left( \widehat{K}(\xi) + \widehat{F}(\xi) -\lambda\right)_+ \diff\xi
- \Intr_{C\leq |\xi| <\left(\frac{2m}{\lambda}\right)^{1/\alpha}} \frac{M}{|\xi|^{\alpha +2}} \diff \xi.
\end{align}
Adding up (\ref{int1}) and (\ref{int2}), we have obtained
\begin{align} \label{boundbelow}
I_1
\geq \Intr_{\RR^d} \left( \widehat{K}(\xi) + \widehat{F}(\xi) -\lambda\right)_+ \diff\xi
+ O(\lambda^{-\frac{d-\alpha-2}{\alpha}}),
\end{align}
as $\lambda\to 0$. Similarly, we could bound $I_1$ from above, so  inequality (\ref{boundbelow}) is in fact an equality.  

\medskip
\noindent
Exactly the same method could be applied to $I_2$ where we obtain
\begin{align*}
I_2 
 = \Intr_{\RR^d} \left( -(\widehat{K}(\xi) + \widehat{F}(\xi))  -\lambda\right)_+ \diff\xi
+ O(\lambda^{-\frac{d-\alpha-2}{\alpha}}),
\end{align*}
and thus
\begin{align*}
I 
 = \Intr_{\RR^d} \left( \left|\widehat{K}(\xi) + \widehat{F}(\xi)\right|  -\lambda\right)_+ \diff\xi
+ O(\lambda^{-\frac{d-\alpha-2}{\alpha}})
\end{align*}
as $\lambda\to 0$. We are now left to compute the integral appearing in this expression. Using polar coordinates $(r,\theta)$, this becomes
\begin{align*}
\Intr_{\RR^d} &\left( \left|\widehat{K}(\xi) + \widehat{F}(\xi)\right| -\lambda\right)_+ \diff\xi
\\
&\qquad = \Intr_{\mathbb{S}^{d-1}}\Intr_0^\infty \left( \left|\frac{f(\theta)}{r^\alpha}+\frac{g(\theta)}{r^{\alpha+1}}\right|-\lambda\right)_+ r^{d-1} \diff r\diff\sigma(\theta).
\end{align*}
We will use the following Lemma to compute this integral.

\begin{lem}
Let $C_1,C_2\in\RR$ be constants, and $\mu>0$ a variable which will be allowed to tend to 0. Then
\begin{equation} \label{lemmaest}
\begin{aligned} 
\Intr_0^\infty &\left(\left|\frac{C_1}{r^\alpha}+\frac{C_2}{r^{\alpha+1}}\right| - \mu \right)_+ r^{d-1} \diff r 
=\frac{\alpha}{d(d-\alpha)}|C_1|^\frac{d}{\alpha} \mu^{1-\frac{d}{\alpha}} 
\\
&\qquad + \frac{1}{d-\alpha-1}\sgn(C_1)|C_1|^\frac{d-\alpha-1}{\alpha}C_2 \mu^{1-\frac{d-1}{\alpha}} + o(\mu^{1-\frac{d-1}{\alpha}})
\end{aligned}
\end{equation}
as $\mu\to 0$. 
\end{lem}

\begin{proof}
Let $h:(0,\infty)\to\RR$ be defined by $h(r)=\frac{C_1}{r^\alpha}+\frac{C_2}{r^{\alpha+1}}$. We first need to find for which values of $r$ we have $h(r)\geq \mu$ and $h(r)\leq -\mu$. 

We distinguish a number of cases depending on the sign of the constants $C_1$ and $C_2$. The case $C_1=0$ is immediate. 

If $C_1>0$ and $C_2<0$ (the case $C_1<0$ and $C_2>0$ is very similar), then the function $h$ increases from $-\infty$ up to a positive value and then decreases to $0$. The equation $h(r)=-\mu$ has one real solution $r^-(\mu)$, and the equation $h(r)=\mu$ has, for $\mu$ small enough, exactly two real solutions, say $r^+_1(\mu)<r^+_2(\mu)$ (see Figure 1). Then $h^{-1}((-\infty,-\mu])=(0,r^-(\mu)]$, and $h^{-1}([\mu,\infty))=[r^+_1(\mu),r^+_2(\mu)]$. These roots can be estimated as follows
\begin{align*}
r^-(\mu)
&=-\frac{C_2}{C_1} - \frac{1}{C_1} \left( -\frac{C_2}{C_1}\right)^{\alpha+1} \mu + o(\mu)
\\
r^+_1(\mu)
&=-\frac{C_2}{C_1} + \frac{1}{C_1} \left( -\frac{C_2}{C_1}\right)^{\alpha+1} \mu + o(\mu)
\\
r^+_2(\mu)
&=C_1^{1/\alpha} \mu^{-1/\alpha}+ \frac{C_2}{\alpha C_1} + o(1)
\end{align*}
as $\mu\to 0$. A straightforward (yet rather tedious) computation then gives (\ref{lemmaest}).

If $C_1>0$ and $C_2\geq 0$ (and similarly if $C_1<0$ and $C_2\leq 0$), then the function $h$ is strictly decreasing from $\infty$ to $0$, so the equation $h(r)=\mu$ has a unique solution $r^+(\mu)$, and $h^{-1}([1,\infty))=[r^+(\mu),\infty)$ (see Figure 2). We can estimate the root 
$$ r^+(\mu) = C_1^{1/\alpha}\mu^{-1/\alpha} + \frac{C_2}{\alpha C_1}  + o(1)$$
as $\mu\to 0$ and (\ref{lemmaest}) follows easily.

\begin{figure}
    \centering
    \begin{minipage}{0.5\textwidth}
        \centering
        \begin{tikzpicture}[yscale=10]
\draw [thick, help lines, ->] (0,0) -- (4,0);
\draw [thick, help lines, ->] (0,-0.2) -- (0,0.2);
\draw [ultra thick, smooth,domain=0.87:4] plot (\x, {pow(\x,-2)-pow(\x,-3)});
\draw [dashed] (0,0.1) -- (4,0.1);
\draw [dashed] (0,-0.1) -- (4,-0.1);
\node [left] at (0,0.2) {$h(r)$};
\node [right] at (4,0) {$r$};
\draw [fill] (1.15345,0.1) ellipse  (0.1 and 0.01);
\draw [fill] (2.42362,0.1) ellipse  (0.1 and 0.01);
\draw [fill] (0.94,-0.1) ellipse  (0.1 and 0.01);
\node [left] at (1.3,0.14) {$r^+_1(\mu)$};
\node [right] at (2.4,0.14) {$r^+_2(\mu)$};
\node [right] at (1,-0.14) {$r^-(\mu)$};
\node [left] at (0,0.1) {$\mu$};
\node [left] at (0,-0.1) {$-\mu$};
\end{tikzpicture}
        \caption{The case $C_1>0, C_2<0$.}
    \end{minipage}\hfill
    \begin{minipage}{0.5\textwidth}
        \centering
        \begin{tikzpicture}[yscale=1]
\draw [thick, help lines, ->] (0,0) -- (4,0);
\draw [thick, help lines, ->] (0,0) -- (0,4);
\draw [ultra thick, smooth,domain=0.77:4] plot (\x, {pow(\x,-2)+pow(\x,-3)});
\draw [dashed] (0,1) -- (4,1);
\node [left] at (0,4) {$h(r)$};
\node [right] at (4,0) {$r$};
\draw [fill] (1.3247,1) circle  (0.1);
\node [right] at (1.3247,1.4) {$r^+(\mu)$};
\node [left] at (0,1) {$\mu$};
\end{tikzpicture}
        \caption{The case $C_1>0 ,C_2\geq 0$.}
    \end{minipage}
\end{figure}
\end{proof}

Using this Lemma we have
\begin{multline*}
\Intr_{\RR^d}  \left( \left|\widehat{K}(\xi) + \widehat{F}(\xi)\right| -\lambda\right)_+ \diff\xi
= \frac{\alpha}{d(d-\alpha)}\lambda^{1-\frac{d}{\alpha}}\Intr_{\mathbb{S}^{d-1}} |f(\theta)|^\frac{d}{\alpha} \diff\sigma(\theta) 
\\
\quad +\frac{1}{d-\alpha-1} \lambda^{1-\frac{d-1}{\alpha}} \Intr_{\mathbb{S}^{d-1}} \sgn(f(\theta)) |f(\theta)|^\frac{d-\alpha-1}{\alpha} g(\theta)\diff\sigma(\theta)
+o(\lambda^{1-\frac{d-1}{\alpha}}).
\end{multline*}
The first term on the right hand side of this equation can be simplified using
\begin{align*}
\Intr_{\RR^d} \left( |\widehat{K}(z)| - 1\right)_+ \diff z
&= \Intr_{\mathbb{S}^{d-1}} \Intr_0^\infty \left( \frac{|f(s)|}{r^{\alpha}}-1\right)_+ r^{d-1} \diff r \diff \sigma(s)
\\
&=\Intr_{\mathbb{S}^{d-1}} \Intr_0^{|f(s)|^{1/\alpha}} (|f(s)|r^{d-\alpha-1}- r^{d-1}) \diff r \diff \sigma(s)
\\
&=\frac{\alpha}{d(d-\alpha)} \Intr_{\mathbb{S}^{d-1}} |f(s)|^{d/\alpha} \diff \sigma(s).
\end{align*}

\medskip
\noindent
This completes the proof of Theorem \ref{lowerbound}.
\end{proof}

\section{Applications}

Let us consider a special case of spherically symmetric kernels
$$
\widehat{K}(\xi) = |\xi|^{-\alpha},
$$
so
\begin{equation*}
K(z)=C|z|^{-(d-\alpha)},
\end{equation*}
where $C=\pi^{-d/2}2^{-\alpha}\frac{\Gamma\left(\frac{d-\alpha}{2}\right)}{\Gamma\left(\frac{\alpha}{2}\right)}$.

%--------------------------
\begin{cor}
Let $ \widehat{K}(\xi) = |\xi|^{-\alpha}$, $0<\alpha<d$. Then
\begin{equation}\label{module}
\sum_k( |\lambda_k| - \lambda)_+ \le \frac{|\Omega|}{(2\pi)^d} \frac{\alpha}{d(d-\alpha)} |\mathbb S^{d-1}| \lambda^{1-\frac{d}{\alpha}} .
\end{equation}
If, moreover, $\alpha<d-1$, we also have the lower bound
\begin{multline*}
\sum_k (|\lambda_k| - \lambda)_+ \geq \frac{|\Omega|}{(2\pi)^d} \frac{\alpha}{d(d-\alpha)} |\mathbb S^{d-1}| \lambda^{1-\frac{d}{\alpha}}
\\
+ \frac{1}{(2\pi)^d} \frac{\Gamma\left(\frac{\alpha+1}{2}\right)\Gamma\left(\frac{d-\alpha}{2}\right)}{\Gamma\left(\frac{\alpha}{2}\right)\Gamma\left(\frac{d-\alpha+1}{2}\right)} \lambda^{1-\frac{d-1}{\alpha}}\int_{\mathbb{S}^{d-1}} A_\Omega(\theta) \diff\sigma(\theta)
+o(\lambda^{1-\frac{d-1}{\alpha}})
\end{multline*}
as $\lambda\to 0$.

\end{cor}

%---------------------
\begin{proof}
Using Theorem \ref{Th1} we find 
\begin{multline*} 
\sum_k( |\lambda_k| - \lambda)_+ \le (2\pi)^{-d}\,  |\Omega| \,\,\lambda^{1-\frac{d}{\alpha}} \,  \int_{\mathbb R^d} (|\xi|^{-\alpha} -1)_+\, d\xi\\
=
\frac{|\Omega|}{(2\pi)^d} \frac{\alpha}{d(d-\alpha)} |\mathbb S^{d-1}\, | \lambda^{1-\frac{d}{\alpha}}. 
\end{multline*} 

For the lower bound, keeping the notation from the previous section, we first need to compute the constant $\gamma=\int_{\mathbb{S}^{d-1}} g(\theta) \diff\theta$. Consider the function $E(z)=e^{-|z|^2/2}$, so $\widehat{E}(\xi)=(2\pi)^{d/2}E(\xi)$. By Parseval's theorem we have
\begin{align*}
(2\pi)^d \int_{\RR^d} F(z)E(z) \diff z = \int_{\RR^d} \widehat{F}(z)\widehat{E}(z) \diff\xi.
\end{align*}
Using polar coordinates $z=r\theta$ on both sides this becomes
\begin{multline*}
(2\pi)^{d/2} \frac{C}{|\Omega|} \int_{\mathbb{S}^{d-1}} A_\Omega(\theta) \diff\sigma(\theta) \int_0^\infty r^\alpha e^{-r^2/2} \diff r 
\\
= \int_{\mathbb{S}^{d-1}} g(\theta) \diff\sigma(\theta) \int_0^\infty r^{d-\alpha-2} e^{-r^2/2} \diff r. 
\end{multline*}
Changing the variable $y=r^2/2$ and using the definition of the gamma function, we finally obtain
\begin{equation*}
\gamma = \frac{2}{|\Omega|}\frac{\Gamma\left(\frac{\alpha+1}{2}\right)\Gamma\left(\frac{d-\alpha}{2}\right)}{\Gamma\left(\frac{\alpha}{2}\right)\Gamma\left(\frac{d-\alpha-1}{2}\right)} \int_{\mathbb{S}^{d-1}} A_\Omega(\theta) \diff\sigma(\theta),
\end{equation*}
and the bound follows from Theorem \ref{lowerbound}.
\end{proof}

%---------------------

\begin{rem}
Note that if $\alpha = 2$, $d\ge3$,  then $0<\alpha<d$, kernel $K(x)$ is the fundamental solution for the Laplacian in $\mathbb R^d$.  We obtain
$$
\sum_k( |\lambda_k| - \lambda)_+ \le (2\pi)^{-d}\,  |\Omega|\, |\mathbb S^{d-1}|\, \lambda^{1-\frac{d}{2}} \,  \left(\frac{2}{d (d-2)}\right).
$$
In particular, if $d=3$, then 
$$
\sum_k( |\lambda_k| - \lambda)_+ \le \frac{1}{\sqrt \lambda} \, \frac{1}{12 \, \pi^3 }\, |\Omega|\, |\mathbb S^{2}|.
$$
\end{rem}

\medskip
\noindent
This inequality allows us to obtain a bound on the number of the eigenvalues greater than  $\lambda$. 
%--------------------------
\begin{cor}
Let $ \widehat{K}(\xi) = |\xi|^{-\alpha}$, $0<\alpha<d$. Then $\mathcal K\ge0$, the eigenvalues of $\lambda_k\ge 0$  and for the number of the eigenvalues greater than $\lambda$ of the operator $\mathcal K$ we have
\begin{equation}\label{n(s)}
n(\lambda) = \#\{ k: \lambda_k>\lambda\} \le (2\pi)^{-d}   \,\lambda^{-\frac{d}{\alpha}} |\Omega| \, |\mathbb S^{d-1}|\, \frac{d^{d/\alpha}}{d\, (d-\alpha)^{d/\alpha}}.
\end{equation}
\end{cor}

\begin{proof}
Let 
\begin{equation*}
\chi_\lambda(t)= 
\begin{cases}
1, & t\ge \lambda, \\
0, & 0\le t <\lambda.
\end{cases}
\end{equation*}
Let $\tau < \lambda$. Then clearly $\chi_\lambda(t) \le \frac{(t-\tau)_+}{(\lambda-\tau)}$.
\begin{multline*}
n(\lambda) = \sum_k \chi(\lambda_k) \le \sum_k \frac{(\lambda_k-\tau)_+}{(\lambda-\tau)}\\
\le  (2\pi)^{-d}  \, |\Omega| \, |\mathbb S^{d-1}|\,\frac{\tau^{1-\frac{d}{\alpha}}}{\lambda-\tau}  \, \left(\frac{\alpha}{d(d- \alpha)}\right)
\end{multline*}
Minimising with respect to $\tau$ we find $\tau = \lambda(1-\alpha/d)$ and thus arrive at \eqref{n(s)}.
\end{proof}

%---------------------

Let us consider spectrum of the operator of Dirichlet boundary value problem $-\Delta^{\mathcal D}$ acting in $L^2(\Omega)$, where $\Omega\subset\mathbb R^d$ is a domain finite measure.
\begin{align*}
-\Delta & u(x)  = \nu u(x),\\
& u(x)\Big|_{x\in\partial\Omega} = 0.
\end{align*}
The best known estimate known estimate for the number $N(\nu)$ of the eigenvalues $\{\nu_k\}$ below $\nu$ of this operator follows from the sharp semiclassical inequality for the Riesz means
$$
\sum_k \left(\nu-\nu_k\right)_+\le (2\pi)^{-d} \, |\Omega|\, \nu^{1+d/2}\, \int_{|\xi|<1} (1-|\xi|^2)\, d\xi.
$$
The latter implies (see \cite{L2})
\begin{equation}\label{N(lambda)}
N(\nu) = \#\{k:\, \nu_k<\nu\} \le (2\pi)^{-d} \, |\Omega|\, \nu^{d/2}\, |\mathbb S^{d-1}| \, \, \frac{1}{d} \, \left(\frac{d+2}{d}\right)^{\frac{d}{2}}.
\end{equation}

\medskip
\noindent
We can compare the last estimate with the semiclassical constant that is still open P\'olya conjecture stated for all domains of finite  measure
$$
N(\nu) \le  (2\pi)^{-d}  \, |\Omega| \, \nu^{d/2} \,\int_{|\xi|^2<1} d\xi =  (2\pi)^{-d}  \, |\Omega|\, \nu^{d/2}\, |\mathbb S^{d-1}| \, \, \frac{1}{d}.
$$
Note that if $\alpha = 2$ and $\widehat{K}(\xi) = |\xi|^{-2}$ then the operator $\mathcal K$ is inverse to $-\widetilde{\Delta} $ with some non-local boundary conditions and since the eigenvalues of $-\Delta^{\mathcal D}$ are larger than the eigenvalues of $-\widetilde{\Delta}$  we have 
$$
N(\nu, -\Delta^{\mathcal D}) \le N(\nu,  -\widetilde{\Delta}) \le n(1/\nu).
$$
Thus we obtain

%-------------

\begin{thm} Let $d\ge 3$ and let $\Omega\subset\mathbb R^d$ be a domain of finite measure. Then for the number of the eigenvalues below $\nu$ of the Dirichlet Laplacian we have
\begin{equation}\label{PL}
N(\nu, -\Delta^{\mathcal D}) \le (2\pi)^{-d}   \,\nu^{\frac{d}{2}}\, |\Omega| \, |\mathbb S^{d-1}|\, \frac{1}{d}\, \left(\frac{d}{d-2}\right)^{d/2}.
\end{equation}
\end{thm} 

\begin{rem}
The constant appearing in the right hand side in \eqref{PL} is not as good as in \eqref{N(lambda)}. It must be related to the fact that when considering the Dirichlet boundary problem the Green function for the Laplacian in the whole space has a negative compensating term that is responsible for the Dirichlet boundary conditions. The integral operator $\mathcal K$ with $\widehat K(\xi) = 1/|\xi|^2$ in $L^2(\Omega)$ is the inverse to the Laplacian with some more complicated non-local boundary conditions, see \cite{KS}. 

\end{rem}

\noindent
{\it Acknowledgements}.
AL was supported by the RSF grant No. 18-11-00032.

\bibliographystyle{plain}
\bibliography{ref}

\end{document}